\documentclass{amsart}
\usepackage{euler,amsmath,amssymb,amscd,amsfonts,latexsym,amsxtra}
\usepackage[mathscr]{eucal}
\usepackage{epsfig, graphicx,pinlabel,pdfsync}
\title[Hilbert schemes are degenerate]{Grothendieck-Pl\"ucker images of Hilbert schemes are degenerate}
\date{\today}
\author{Donghoon Hyeon}
\author{ Hyungju Park}
\address{Department of Mathematical Sciences, Seoul National University, Seoul, R. O. Korea \\Tel: +82-2-880-2666, Fax: +82-2-887-4694}
 \email{dhyeon@snu.ac.kr}
 \email{parkhyoungju@snu.ac.kr}
\newtheorem{theorem}{Theorem}[section]
\newtheorem*{theoremquote}{Theorem}

\newtheorem{lemma}[theorem]{Lemma}
\newtheorem{prop}[theorem]{Proposition}
\newtheorem{coro}[theorem]{Corollary}
\theoremstyle{definition}
\newtheorem{defn}[theorem]{Definition}

\def\mbb{\mathbb}

\def\a{\alpha}
\def\b{\beta}

\def\bP{\mbb P}
\def\bZ{\mbb Z}
\def\bN{\mbb N}
\def\bQ{\mbb Q}

\def\GL{{\mathrm GL}}

\def\bP{\mathbb P}
\def\bG{\mathbb G}

\def\inj{\hookrightarrow}

\def\inj{\hookrightarrow}

\def\Gr{Gr}

\def\cO{\mathcal O}

\DeclareMathOperator{\sch}{{\mathrm C}}
\DeclareMathOperator{\hilb}{{\mathrm Hilb}}
\def\ba{{\vec{\a}}}
\def\bbeta{{\vec{\beta}}}
\def\gin{\textup{Gin}\, }

\input xy
\xyoption{all}

\input epsf
\def\bQ{\mathbb Q}

\def\ii{{\bf in}\,}

\newtheorem{remark}[theorem]{Remark}

\begin{document}

\begin{abstract} We study the decompositions of Hilbert schemes induced by the Schubert cell decomposition of the Grassmannian variety and show that Hilbert schemes admit a stratification into locally closed subschemes along which  the generic initial ideals remain the same.
We give two applications: First, we give a completely  geometric proofs of the existence of the generic initial ideals and of their Borel fixed properties. Secondly, we prove that when a Hilbert scheme of nonconstant Hilbert polynomial is embedded by the Grothendieck-Pl\"ucker embedding of a high enough degree, it must be degenerate.

\end{abstract}

\thanks{The authors thank H. Loh for pointing out \cite{Marinari-Ramella} which provides an important ingredient in the proof of Theorem~\ref{T:degenerate}.
We thank Hwangrae Lee who made an observation in Section~\ref{S:other} regarding the relation between our theorem and the work of Notari and Spreafico.
We also greatly benefited from conversations with Young-Hoon Kiem. The first named author was supported by the Research Resettlement Fund for the new faculty of Seoul National University, and the following grants funded by the government of Korea:
NRF grant 2011-0030044 (SRC-GAIA) and NRF grant NRF-2013R1A1A2010649}
\maketitle

\section{Introduction and preliminaries}
In this paper, we give a geometric study of generic initial ideals. Given an ideal $I$ of a polynomial ring $k[x_0,\dots,x_n]$ and a monomial order $\prec$, the generic initial ideal $\gin_\prec(I)$ of $I$ is roughly defined as the monomial ideal generated by the initial terms of $I$ after a generic coordinate change. Its existence and basic properties were first worked out by Galligo \cite{Galligo} in characteristic zero  and subsequent works of Bayer and Stillman \cite{Bayer-Stillman} and of Pardue \cite{Pardue} established fundamental properties in prime characteristic. Generic initial ideals found useful applications in the study of Hilbert schemes \cite{Hartshorne}, and in the study of the Castelnuovo-Mumford regularity and the complexity of Gr\"obner basis computation \cite{BS} just to name a few.

We shall take Green's geometric viewpoint of initial ideals \cite{Green} and prove further properties about generic initial ideals. Understanding the geometry of initial ideals leads to a more conceptual and geometric proof of the existence of the generic initial ideals (Proposition~\ref{P:sch-decompII} and Definition~\ref{D:gin}). We also obtain a completely geometric proof of the Borel fixedness (Proposition~\ref{P:Borel}), which is the most important combinatorial property of generic initial ideals. In essence, it is not a new proof but more of a reformulation since it shares with the algebraic proof the key component, which is considering  the non-vanishing of the coefficient of the largest Pl\"ucker  monomial. Nonetheless, we do believe that our geometric reformulation is a better display of the essence of the proof, and it has the obvious advantage of being terse and to the point once we set up the machinery.

We also prove that the Hilbert schemes admit a stratification into locally closed subschemes consisting of ideals with the same generic initial ideals. More precisely,

\begin{theoremquote} There is a finite decomposition
\[
\hilb^P(\bP(V)) = \coprod_{\ba\in\mathcal G} \Gamma_\ba
\]
into locally closed subschemes $\Gamma_\ba = \{ [I]  \, | \, \gin_\prec(I) = I_\ba\}$ where $\ba$ runs through all indices such that $I_\ba$ are Borel fixed.  Moreover, for each irreducible component $H$ of $\hilb^P(\bP(V))$, there is a unique maximal index $\ba_{H} \in \mathcal G$ such that $\Gamma_{\ba_H}$ is Zariski open dense in $H$.
\end{theoremquote}
This will be established in Section~\ref{S:gin-decomposition}. As a corollary, we shall retrieve the main statement of \cite[Theorem~1.2]{Conca}.

Some authors have worked on the stratification of the Hilbert schemes of ideals according to their initial ideals with respect to a  monomial order \cite{Roggero-Terracini, Notari-Spreafico}. Bertone, Lella and Rogero considered in \cite{BLR} what is called the {\it Borel cover} (an {\it open cover}, as opposed to a stratification) of the Hilbert scheme. In Section~\ref{S:other}, we shall briefly sketch these related works and point out the major differences from our work.

As the most prominent application of the geometric study of the stratification, we demonstrate a very  important and fundamental extrinsic geometry of the Hilbert scheme: The  Grothendieck-Pl\"ucker embedding of high enough degree  is degenerate. More precisely,

\begin{theoremquote} Let $P$ be a nonconstant admissible Hilbert polynomial. For any $m \gg 0$ (especially, $m>m_0$), $\phi_{m} (\hilb^P(\bP(V)))$ is degenerate.
\end{theoremquote}
Here,  $m_o$ is the Gotzmann number of the Hilbert polynomial $P$ and $\phi_m$ is the Grothendieck-Pl\"ucker embedding (Equation~(\ref{E:Gr-Pl})). Admissible Hilbert polynomials are the Hilbert polynomials of graded ideals.  It is well known that for any admissible Hilbert polynomial $P$, there exists a lex-initial ideal whose Hilbert polynomial is $P$, and this completely classifies the admissible Hilbert polynomials: See, for instance, \cite[Appendix~C, P.299]{IK}.  This theorem will be proved in Section~\ref{S:degenerate}.

\

We work over an algebraically closed field $k$ of characteristic zero.

\section{Schubert decomposition of the Hilbert schemes}\label{S:sch-hilb}

\subsection{Schubert cells in the Grassmannian} To introduce various notations properly, and for the sake of completeness, we recapitulate the Schubert cell decomposition of Grassmannian varieties.
Let $d < n$ be positive integers, and $E$ be a $k$-vector space with an ordered basis $\{e_\a\}_{\a \in \mathcal A}$. Here, $\mathcal A$ is an index set and the order is denoted by $\prec$. We also let $\prec$ denote the induced order on $\mathcal A$.
The {\it standard Borel subgroup} $B \subset \GL(E)$ consists of $g \in \GL(E)$ such that $g.e_\a = \sum_{i=1}^{n} g_{\a\b}e_\b$ and $g_{\a\b}=0$ for all $\b \succ \a$.

Let $\Gr_dE$ be the Grassmannian variety of $d$-dimensional subspaces of $E$ and $\ba = (\a(1),\dots,\a(d)) \in \mathcal A^d$ satisfying $e_{\a(i)} \succ  e_{\a(i+1)}$.  The {\it Schubert cells} are defined to be the $B$-orbits of the $d$-dimensional coordinate subspaces i.e. for any $\ba$ as above,
\[
\sch_\ba = B.E_{\ba}
\]
where $E_\ba$ is the subspace spanned by $e_{\a(1)}, \dots, e_{\a(d)}$.

For our purpose, it is useful to have the following description of Schubert cells in terms of the initial subspace. A {\it monomial} of $\bigwedge^d E$ is an  element of the form
\[
e_\ba:=e_{\a(1)}\wedge \cdots \wedge e_{\a(d)}
\]
with $\a(i)\succ \a(i+1)$, and we order the monomials lexicographically.

For any $v = \sum_\a a_\a e_\a \in E$, the {\it initial vector} $\ii_\prec(v)$ is simply $e_\b$ such that $a_\b \ne 0$ and $a_\a = 0$ for all $e_\a \succ e_\b$. Let $F \subset E$ be a $d$-dimensional subspace of $E$. Then the {\it initial subspace} $\ii_\prec(F)$ is defined to be the subspace spanned by $\ii_\prec(w)$, $\forall w\in F$.

For $\ba = (\a(1),\dots,\a(d))$ with $\a(i) \succ \a(i+1)$, The $\ba$th Pl\"ucker coordinate $p_\ba(F)$ of $F$ is  the $e_{\a(1)}\wedge \cdots \wedge e_{\a(d)}$-coefficient of $\bigwedge^d F$.
Then the  $\ba$th Schubert cell is precisely
\begin{equation}\label{E:Schubert-cell}
\sch_\ba := \{ F \in \Gr_d(E) \, | \, p_\ba(F) \ne 0, \, p_{\ba'}(F) = 0, \, \forall \ba' \succ \ba\}.
\end{equation}

We define the partial order $\prec_s$ on $\mathcal A^d$ as follows: For any two indices $\ba$ and $\ba'$, $\ba \prec_s \ba'$ if and only if $\a(i) \prec \a'(i)$ for all $i$. Then the Schubert cells are partially ordered accordingly, and the closure $\overline{C_\ba}$, called the Schubert variety, is the union $\coprod_{\ba'\preceq_s \ba} C_{\ba'}$. We point the readers to the excellent lecture note by Michel Brion at \verb"https://www-fourier.ujf-grenoble.fr/~mbrion/notes.html".

\subsection{ Decomposition of the Hilbert schemes induced by the Schubert cells of the Grassmannians}\label{S:Hilb}

Let $V$ be a $k$-vector space of dimension $n+1$ and $x_0, \dots, x_n \in V^*$ be a basis of the dual vector space. The symmetric product $S^mV^*$ has a basis consisting of degree $m$ monomials
\[
x^\a := x_0^{\a_0}\cdots x_n^{\a_n}
\]
where $\a = (\a_0,\dots,\a_n) \in \bN^{n+1}$ has component sum $|\a| = \sum \a_i = m$.
Let $\succ$  be a monomial order, and let $B \subset \GL(V^*)$ and $B' \subset \GL(S^mV^*)$ be the standard Borel subgroups with respect to $\succ$.
We abuse the notation and let $\succ$ also denote the induced monomial order on $\bZ_{\ge 0}^{n+1}$ i.e. $\a \succ \b$ if and only if $x^\a \succ x^\b$.

\begin{defn} $\rho_m : \GL(V^*) \to \GL(S^mV^*)$ denotes the natural homomorphism defined
\[
\rho_m(g) x^\a = \prod_{i=0}^n (g.x_i)^{\a_i}.
\]
\end{defn}

\begin{lemma} $\rho_m(B) \subset B'$.
\end{lemma}
\begin{proof} For any $g \in B$, we have
\[
\begin{array}{ccl}
g.x^\a & = & g.\prod_i x_i^{\a_i} = \prod_i (g_{ii}x_i + l.o.t.s)^{\a_i} \\
& = & \prod_i (g_{ii}^{\a_i}x_i^{\a_i} + l.o.t.s)\\
& = & \left(\prod_i g_{ii}^{\a_i}\right) x^\a + l.o.t.s
\end{array}
\]
\end{proof}

Let $P \in \bQ[m]$ be a rational polynomial admissible in the sense of \cite[Theorem~1.3]{Valla} and $Q(m) = \dim_k S^mV^* - P(m) = \binom{n+m}{m} - P(m)$.
There is a number $m_0$, called the Gotzmann number, such that for all $m \ge m_0$, any homogeneous ideal $I \subset S:=k[x_0,\dots,x_n]$ with Hilbert polynomial $P$ is $m$-regular \cite{Gotzmann}. This implies that we have an exact sequence
\begin{equation}\label{E:hilbertpoint}
0 \to I_m \to \Gamma(\bP(V), \cO_{\bP(V)}(m)) \to \Gamma(X, \cO_X(m)) \to 0
\end{equation}
where $X \subset \bP(V)$ is the closed subscheme of $\bP(V)$ cut out by $I$.  The point in $\Gr_{Q(m)}S^mV^*$ defined by the equation (\ref{E:hilbertpoint}) is called the {\it $m$th Hilbert point of $I$ (or, of $X$)}, and is denoted by $[I]_m$ or $[X]_m$.
Gotzmann's theorem implies that we have an embedding of the Hilbert scheme
\begin{equation}\label{E:Gr-Pl}
\begin{array}{ccc}
\phi_m : \hilb^P\bP(V) & \to & \Gr_{Q(m)}S^mV^* \\
\left[I\right] & \mapsto & [I]_m
\end{array}.
\end{equation}
We call $\phi_m$ the $m$th {\it Grothendieck-Pl\"ucker embedding} of $\hilb^P \bP(V)$. Fix $m \ge m_0$, let $d := Q(m)$ and consider the Schubert cell decomposition of $\Gr_d S_m$ where $S_m$ is the degree $m$ part $S^mV^*$ of $S$. The Schubert cells $\sch_\ba$ are indexed by $\ba = (\a(1),\dots,\a(d))$ where each $\a(i) \in \bN^{n+1}$ has component sum $m$ and $x^{\a(i)} \succ x^{\a(i+1)}$.  For convenience, we say that such an {\it $\ba$ is an index}:
\[
\sch_\ba = \{ F \subset S_m \, | \, \ii_\prec(F) = k\langle x^{\a(1)},\dots,x^{\a(d)} \rangle\}.
\]

\begin{defn} Let $\ba$ be an index. We define $I_\ba$ to be the saturation  of the ideal generated by $x^{\a(1)}, \dots, x^{\a(d)}$:
 \[
 I_\ba := (\langle x^{\a(1)},\dots,x^{\a(d)}\rangle : \langle x_0,\dots,x_n\rangle^\infty).
 \]
\end{defn}
\begin{lemma}\label{L:Iba} Let $J$ be a saturated homogeneous ideal of $k[x_0,\dots,x_n]$ with Hilbert polynomial $P$. If $J_m \in \sch_\ba$, then $\ii_\prec J = I_\ba$.
\end{lemma}
\begin{proof}  Since the Hilbert polynomial of $\ii_\prec J$ is $P$ and  $m$ is not smaller than the Gotzmann number $m_0$ of $P$, $(\ii_\prec J)_{m+l} = S_l (\ii_\prec J)_m$ for all $l \ge 0$. Hence $(\ii_\prec J)_m = (I_\ba)_m$ for all $m \ge m_0$ and it follows that $\ii_prec J = I_\ba$ since both ideals are saturated.
\end{proof}

\begin{lemma} \label{L:pushup} Let $J$ be as in the previous lemma. Suppose $J_m \in \sch_\ba$. Then there exists a nonempty open subscheme $U \subset \GL(V^*)$ such that for all $g \in U$, $\ii_\prec(g.I) = I_{\ba'}$ for some $\ba' \succeq \ba$.
\end{lemma}

\begin{proof} Let $U \subset \GL(V^*)$ be the open subscheme that is complement to the closed subscheme cut out by the Pl\"ucker equation  $p_\ba(g.J_m) \ne 0$. Since $p_\ba(J_m) \ne 0$, $U$ is nonempty. That  $U$ has the desired property is clear from the defining property (\ref{E:Schubert-cell}) of the Schubert cells and Lemma~\ref{L:Iba}.
\end{proof}

\[
\begin{array}{cccl}
\Psi_m : & \GL(V^*) \times \hilb^P(\bP(V)) & \to & \Gr_dS^mV^* \\
& (g,[I]) & \mapsto & [g.I]_m
\end{array}
\]
By abusing terminology, we shall call $\Psi_m$ the $m$th Grothendieck-Pl\"ucker embedding when there is no danger of confusion. Note that $[g.I]_m = \rho_m(g).[I]_m$, which amounts to say that $\Psi_m$ is $\GL(V^*)$-equivariant where $\GL(V^*)$ acts on $\GL(V^*) \times \hilb^P(\bP(V))$ on the first factor, and on $\Gr_dS^mV^*$ through $\rho_m$.

\begin{defn} $\sch'_{\ba,m} = (\GL(V^*) \times \hilb^P(\bP(V)) )\times_{\Gr_dS^mV^*} \sch_{\ba,m}$.
\end{defn}

\begin{remark} When there is no danger of confusion, we suppress the subscript $m$.
\end{remark}

The following two lemmas are immediate from Lemma~\ref{L:Iba}:
\begin{lemma}\label{L:in}  $(g,[I]) \in \sch'_\ba$ if and only if $\ii_\prec(g.I) = I_\ba$.
\end{lemma}

\begin{lemma}\label{L:B-invar} $\sch'_\ba$ is Borel invariant i.e. $B.\sch'_\ba = \sch'_\ba$.
\end{lemma}
\begin{proof} For $b \in B$ and $(g, [I]) \in \sch'_\ba$, we have
\[
\Psi_m(b.(g,[I])) = \rho_m(b).\Psi_m((g,[I])) \in \rho_m(b).B'.E_\ba = B'.E_\ba.
\]
\end{proof}

\begin{lemma} \label{L:basic-cell} Let $\mathcal I$ be a set of indices and $X \subset \cup_{\ba \in \mathcal I} C_\ba$ be an irreducible subset. Let $\ba^*$ be a maximal index such that $C_{\ba^*}\cap X \ne \emptyset$. Then $ C_{\ba^*}\cap X$ is open in $X$. Consequently, such $\ba^*$ is unique.
\end{lemma}

\begin{proof} Let $\ba_1, \dots, \ba_t$ be maximal indices whose Schubert cells meet $X$. Then, by definition of maximality,
$X \subset \coprod_{i=1}^t \coprod_{\bbeta \preceq \ba_i} C_{\bbeta} = \coprod_{i=1}^t \overline{C_{\ba_i}}$. Note that each $\overline{C_{\ba_i}}$ is connected, closed and open in $\coprod_{i=1}^t \overline{C_{\ba_i}}$ (since $\coprod_{i=1}^t \overline{C_{\ba_i}} \setminus \overline{C_{\ba_j}} = \coprod_{i\ne j} \overline{C_{\ba_i}}$).
Since $C_{\ba_i}$ is open in its closure, it follows that $C_{\ba_i}\cap X$ is open in $X$. Since $X$ is irreducible, it follows that $t = 1$.
\end{proof}

We summarize our findings in a proposition:

\begin{prop}\label{P:sch-decompI} Let $H$ be an irreducible component of $\hilb^P(\bP(V))$. Then there is a finite decomposition of $\GL(V^*) \times H$ into nonempty locally closed subschemes
\[
\GL(V^*)\times H = \coprod_{\ba\in \mathcal J} \sch'_\ba
\]
such that
\begin{enumerate}
\item if $\ba^\star$ is a maximal dimensional cell (such that $\sch'_\ba \ne \emptyset$), $\sch'_{\ba^\star}$ is Zariski open dense in $\GL(V^*)\times H$;
\item $(g,[I])$ and $(g',[I'])$ are in the same $\sch'_\ba$ if and only if $\ii_\prec(g.I) = \ii_\prec(g'.I')$;
\item each $\sch'_\ba$ is $B$-invariant.
\end{enumerate}
\end{prop}

\begin{proof} In a union of Schubert cells, any cell of maximal dimension is open. The index set $\mathcal J$  consists of $\ba$ such that $$\Psi_m(\GL(V^*)\times H)\cap \sch_\ba  = \phi_m(H)\cap \sch_\ba\ne \emptyset.$$
Since $\GL(V^*)\times H$ is irreducible, there is a unique $\ba^*$  such that the corresponding Schubert variety  $\overline{C_{\ba^*}}$ contains it. This establishes the first item. The second and the third items are precisely the Lemmas \ref{L:in} and \ref{L:B-invar}.
\end{proof}

We obtain an induced decomposition of the irreducible components of  Hilbert schemes, simply by taking the trivial slice $\{1\}\times H$ of the product $\GL(V^*)\times \hilb^P(\bP(V))$.  We retrieve the following result of Notari and Spreafico:

\begin{coro}\label{C:in-decomp}\cite[Theorem~2.1]{Notari-Spreafico} Fix a monomial order $\prec$ on the set of monomials of $k[x_0,\dots, x_n]$ and a monomial ideal $I_0 \subset k[x_0,\dots,x_n]$. Then there exists a locally closed subscheme $H_{I_o}$ of the Hilbert scheme $\hilb^P(\bP(V))$ whose closed points are in bijective correspondence with the saturated ideals of $k[x_0,\dots, x_n]$ whose initial ideal equals $I_0$.
\end{coro}

\subsection{Gin decomposition of the Hilbert scheme} \label{S:gin-decomposition}

\begin{theorem}\label{T:gin-decomp} There is a finite decomposition
\[
\hilb^P(\bP(V)) = \coprod_{\ba\in\mathcal G} \Gamma_\ba
\]
into locally closed subschemes $\Gamma_\ba = \{ [I]  \, | \, \gin_\prec(I) = I_\ba\}$ where $\a$ runs through all Borel fixed ideals. Moreover, for each irreducible component $H$ of $\hilb^P(\bP(V))$, there is a unique maximal index $\ba_{H} \in \mathcal G$ such that $\Gamma_{\ba_H}$ is Zariski open dense in $H$.
\end{theorem}

\begin{proof}
We give an inductive proof. Let $m$ be an integer larger than the Gotzmann number of $P$.
Set $\Gamma_{0j} = \emptyset$, $j \in \bN$. Suppose that we have constructed $\Gamma_{11}, \dots, \Gamma_{1s_1}, \dots, \Gamma_{\ell-1 \, 1}, \dots, \Gamma_{\ell-1 s_{\ell-1}}$ such that for each irreducible component $H_{uj}$ of
\[
Z_{u}:= \hilb^P(\bP(V))\setminus \coprod_{i\le u-1} \Gamma_{ij}, \quad u \le \ell-1
\] there exists a unique $\Gamma_{uj}$ which is  open dense in $H_{uj}$ and an index $\ba^*_{uj}$ such that $\gin_\prec I = I_{\ba^*_{uj}}$ for all $I \in \Gamma_{uj}$.

Let $H_{\ell 1}, \dots, H_{\ell \, s_{\ell}}$ be the irreducible components of $Z_\ell$ defined as above.
Let $\pi_2$ denote the projection from $\GL(V^*)\times H_{\ell j} $ to the  second factor.
Since $\GL(V^*)\times  H_{\ell j}$ is irreducible, by Lemma~\ref{L:basic-cell} there exists a unique maximal index $\ba^*_{\ell j}$ such that $\Psi_m(\GL(V^*)\times H_{\ell j}) \cap C_{\ba^*_{\ell j}}$ is non-empty open in $\Psi_m(\GL(V^*)\times H_{\ell j}))$. Let $U_{\ba^*_{\ell j}}$ be the fibre product $C_{\ba^*_{\ell j}} \times_{\Gr_dS^mV^*} \left(\GL(V^*)\times H_{\ell j}\right)$.
 It is an open subscheme of $\GL(V^*)\times H_{\ell j}$, and its projected image $\Gamma_{\ell j} = \pi_2(U_{\ba^*_{\ell j}})$  is an open subscheme of $H_{\ell j}$ since projections are flat.

For any $[I] \in \Gamma_{\ell j}$, $C_{\ba^*_{\ell j}} \times_{\Gr_dS^mV^*} (\GL(V^*)\times \{[I]\})$ is not empty, and  open  in $\GL(V^*)\times \{[I]\}$ which we identify with $\GL(V^*)$.
Clearly, $\ba^*_{\ell j}$ is the maximal index whose Schubert cell meets $\Psi_m(\GL(V^*)\times \{[I]\})$.
Hence for any $g$ in the {\it open nonempty} subscheme $C_{\ba^*_{\ell j}} \times_{\Gr_dS^mV^*} (\GL(V^*)\times \{[I]\})$ of $\GL(V^*)$, we have $[\ii_\prec(g.I)]_m = [I_{\ba^*_{\ell j}}]_m$, and since $m$ is at least as large as the Gotzmann number, $\ii_\prec(g.I) = I_{\ba^*_{\ell j}}$. That is $\gin_\prec(I) = I_{\ba^*_{\ell j}}$ for any $[I] \in \Gamma_{\ell j}$, and we rename  $\Gamma_{\ell j}$ to  $\Gamma_{\ba^*_{\ell j}}$ and obtain the statement of the theorem.
\end{proof}

\begin{remark} The definition/construction of the locally closed subschemes $\Gamma_{ij}$ depends on the choice of the embedding $\phi_m$ of the Hilbert scheme but their properties determine them uniquely.
\end{remark}

As a corollary, we retrieve the following. Let $\prec$ be a monomial order and $P$ be an admissible Hilbert polynomial.
We assume that $\hilb^P\bP^n$ is embedded in a suitable Grassmannian.
Recall that an {\it initial segment} in degree $d$ and length $\ell$ with respect to $\prec$ is simply the set of the first $\ell$ monomials of degree $d$.
\begin{coro}\cite[Theorem~1.2]{Conca} For any general member $I$ of an  irreducible component $H$ of $\hilb^P\bP^n$, we have
\[
\gin_\prec(I) = I_{\ba^*}
\]
where $\ba^*$ is the maximal index such that $H$ meets $C_{\ba^*}$. In particular, the generic initial ideal of general points in the plane equals the ideal which is generated by initial segments in every degree.
\end{coro}

\begin{proof} The first statement is straight from Theorem~\ref{T:gin-decomp} and its proof. The second statement follows since a Hilbert scheme of points on a smooth surface is smooth and irreducible, and there exists an ideal with Hilbert polynomial $P$ that is generated by initial segments in all degrees \cite[Lemma~5.5]{Conca}. Note that the assertion does not depend on the embedding $\phi_m$ since $I_{\ba^*}$ remains the same by Lemma~\ref{L:Iba}.
\end{proof}

\begin{remark} Although we retrieve the main statement of Theorem~1.2 of \cite{Conca}, Conca and Sidman do more: They
explicitly give a set of conditions on the points that guarantee that the generic initial ideal is the initial segment ideal.
\end{remark}
%\begin{example} (Hilbert scheme of $n$ points in the plane) \end{example}

\section{Primary and secondary generic initial ideals}
We retain the notations from the previous section. As an application of our geometric study of the Gin decomposition of the Hilbert scheme, we give a  geometric proof of the existence of generic initial ideals and their Borel-fixed properties. One of the key ingredients is that initial ideals can be thought of as flat limits with respect to a one-parameter subgroup action: Dave Bayer and Ian Morrison used this in their study of state polytopes of Hilbert points \cite{BS}, and more recently Morgan Sherman has also used it to prove that the one-parameter subgroup  \cite{Sherman} taking an ideal to its generic initial ideal is also Borel fixed.

Fix a saturated ideal $I \subset k[x_0,\dots,x_n]$ with Hilbert polynomial $P$, and consider the orbit map
\[
(\Psi_m)_I : \GL(V^*) \simeq \GL(V^*)\times [I] \inj \GL(V^*) \times \hilb^P(\bP(V)) \stackrel{\Psi_m}{\to} \Gr_dS^mV^*.
\]
In short, $(\Psi_m)_I(g) = [g.I]_m$. We have the induced decomposition
\[
\GL(V^*) \simeq \GL(V^*)\times [I] = \coprod_\ba (\GL(V^*)\times [I]) \cap \sch'_\ba.
\]
We let $\sch''_\ba$ denote $(\GL(V^*)\times [I]) \cap \sch'_\ba$ and regard it as a locally closed subscheme of $\GL(V^*)$. From Proposition~\ref{P:sch-decompI} and its proof,  we easily obtain

\begin{prop}\label{P:sch-decompII} There is a finite decomposition of $\GL(V^*)$ into locally closed subschemes
\[
\GL(V^*) = \coprod_{\ba} \sch''_\ba
\]
such that
\begin{enumerate}
\item if  $\ba^\star$ is the maximal index (such that $\sch''_\ba \ne \emptyset$), $\sch''_{\ba^\star}$ is Zariski open dense;
\item $g, g'$ are in the same $\sch''_\ba$ if and only if $\ii_\prec(g.I) = \ii_\prec(g'.I)$;
\item each $\sch''_\ba$ is $B$-invariant.
\end{enumerate}
\end{prop}

\begin{remark}  $C''_{\ba^\star}$ meets the unipotent subgroup $U = \{g \in B \, | \, g_{\a\a} = 1, \forall \a\}$ since $B^oU$ is Zariski open in $\GL(V^*)$. See for instance, \cite[Theorem~15.18]{Eisenbud}.)
\end{remark}

\begin{defn}\label{D:gin} \emph{The (primary) generic initial ideal} of $[I]$ is $\ii_\prec(g.I)$ for any $g \in \sch''_{\ba^\star}$, and it equals $I_{\ba^\star}$. The {\it secondary generic initial ideal} with respect to $\ba \ne \ba^\star$ is $I_\ba = \ii_\prec(g.I)$ for $g \in \sch''_\ba$.
\end{defn}

Let $B^o$ denote the opposite Borel subgroup:
\[
B^o := \{g \in \GL(V^*) \, | \, g. x^\a = \sum c_{\a\b} x^\b, c_{\a\b} = 0, \, \forall \b \prec \a \}.
\]
One sees from the definition of $B$ and $B^o$ that, for any $b \in B$ (resp. $b\in B^o$) and $[I]_m \in C_{\vec\a} \subset \Gr_dS^mV^*$,  $b.[I]_m \in C_{\vec\b}$ with $\b \preceq \a$ (resp. $\b \succeq\a$). We symbolically write
\[
B^o.[I]_m \succeq [I]_m \succeq B.[I]_m.
\]

\begin{prop}\label{P:Borel} \cite{Galligo, Bayer-Stillman, Pardue} The primary generic initial ideals are Borel fixed. That is, $B^o \gin_\prec(I) = \gin_\prec(I)$.
\end{prop}

\begin{proof} Let $[I] \in \hilb^P\bP(V)$ and $\vec\a^\star$ be the maximal index for $[I]_m$ i.e. $I_{\ba^\star} = \ii_\prec(g.I)  = \gin_\prec(I)$ for any $g \in C''_{\vec{\a}^\star}$. We fix a favorite $g$ and work with it for the rest of this proof. Let $b \in B^o$ and suppose $b.I_{\ba^\star} \ne I_{\ba^\star}$. Since $B^o.[I]_m \succeq [I]_m$,  $b.I_{\ba^\star} \in C_{\vec\b}$ for some $\vec\b \succ \ba$.

There is a one-parameter subgroup $\lambda : \bG_m \to \GL(V^*)$, diagonalized by the basis $\{x_0, \dots, x_n\}$, such that $ \lim_{t\to 0} \lambda(t).g.[I]_m = \ii_\prec(g.[I]_m)$: Due to \cite[Proposition~1.11]{Sturmfels}, there exists $\omega \in \bZ_{\ge 0}^{n+1}$ such that $\ii_\omega (g.I) = \ii_\prec (g.I)$, where $\ii_\omega(I)$ means the ideal generated by the initial {\it forms} $\ii_\omega(f)$ with respect to the partial weight order defined by $\omega$, $\forall f \in I$. Such $\omega$ is obtained by computing a Gr\"obner basis $\mathcal G$ and choosing $\omega$ such that $\ii_\omega(f) = \ii_\prec(f)$ for all $f \in \mathcal G$. Let $\lambda : \bG_m \to \GL(V^*)$ be the 1-PS associated to $-\omega$ i.e. $\lambda(t).x_i = t^{-\omega_i} x_i$. Then $
[\ii_\prec(g.I)] = \lim_{t\to 0} \lambda(t).[g.I]$
in the $\hilb^P(\bP(V))$ \cite[Corollary~3.5]{BM}.

Since $\lambda$ is diagonalized by $\{x_0,\dots,x_n\}$, the Schubert cells are invariant under its action. By our choice of $g \in C''_{\ba^*}$, $g.[I]_m \in C_{\ba^*}$ and hence $\lambda(t).g.[I]_m$ is contained in $C_{\ba^*}$, $\forall t \ne 0$.  Since $\ba^\star \prec \vec\b$, $C_{\ba^\star} \subset Z:=\overline C_{\vec\beta}$ and hence we have $\overline{\lambda(\bG_m).g.[I]_m} \subset Z$. The Schubert cells are locally closed, so $C_{\vec\b}$ is open in $Z$. Since the limit of $\lambda(t).g.[I]_m$ is in the open set $C_{\vec\b}$ (of $Z$), it follows that $\lambda(t).g.[I]_m \in C_{\vec\b}$ for some $t \ne 0$. This contradicts the $\lambda(\bG_m)$-invariance of $C_{\ba^*}$.

\end{proof}

\section{Other stratifications and covers}\label{S:other}
In this section, we shall describe related works and point out the apparent and crucial differences that distinguish our work. Let $\bP(V)$, $k[x_0,\dots, x_n] = \oplus_m S^m V^*$, and $\hilb^P \bP(V)$ be as before in Section~\ref{S:Hilb}.

\subsection{Stratification according to the initial ideals}
The first work appearing in the literature regarding the stratification
\[
\hilb^P \bP(V) = \coprod H_{I_o}
\]
of Hilbert schemes by using initial ideal is \cite{Notari-Spreafico} which we retrieved in Corollary~\ref{C:in-decomp}. This stratification is clearly different from ours. In their stratification, there is a unique stratum for each monomial ideal whereas in ours, there is a unique stratum for each {\it Borel fixed} ideal. Hence the stratification $\hilb^P (\bP(V)) = \coprod H_{I_o}$ has far more strata. Also, a stratum $H_{I_o}$ in general is not contained in one of our strata $\Gamma_{\ba}$ since $\ii_\prec I = \ii_\prec J$ does not imply $\gin_\prec I = \gin_\prec J$: Let $\prec$ be the degree reverse lexicographic order. There are ideals $I$ whose regularity is strictly lower than that of $J = \ii_\prec I$. Then $\ii J = \ii (\ii I) = \ii I$ but $\gin I \ne \gin J$ since the regularity is preserved under taking the generic initial ideal with respect to the degree lexicographic order. This was pointed out to the author by Hwangrae Lee.

Notari and Spreafico studied the properties of the strata and showed that $H_{I_0}$ is isomorphic to an affine space if $H_{I_0}$ is nonsingular at the Hilbert point of $I_0$. They also considered the strata $H_{I_\star}$ that contains an open subset (of an irreducible component $H$) and showed that $I_\star$ should be Borel fixed.
The distinguished open subschemes $H_{I_\star}$ and $\Gamma_{\alpha_H}$ (from Theorem~\ref{T:gin-decomp}) are more closely related than others. First off, the indices $I_\star$ and $\alpha_H$ are both determined by the largest Schubert cell that intersects the Grothendieck-Pl\"ucker image of $H$ (as in Proposition~\ref{P:sch-decompI}), so $I_\star = I_{\alpha_H}$. Also, if $\ii I = I_\star$, then $\gin I = I_\star$ due to Lemma~\ref{L:pushup}. Hence we conclude that $H_{I_\star} = H_{I_{\alpha_H}} \subset \Gamma_{\alpha_H}$. They are not equal in general, as can be easily seen in the hypersurface cases.

\subsection{Borel open cover} In \cite{BLR}, Bertone, Lella and Roggero considered the {\it open cover} $\hilb^P\bP(V) = \bigcup_{g, J} H_J^g$ of the Hilbert scheme, where
\begin{itemize}
\item[i.] the indices $g$ and $J$ respectively run over $\mathrm{PGL}(V^*)$ and the set of all Borel fixed ideals of Hilbert polynomial $P$, and
\item[ii.] the open subscheme $H_J^g$ is the $g$-translate of the open subscheme $\phi_m^{-1}\left(C_{\ba}\right)$ where $\ba$ is the index satisfying $J = I_{\ba}$.
\end{itemize}

Note that $H_{I_\ba}^g$ is an {\it open subscheme} complement to the hypersurface $ \{J \in \hilb^P\bP(V) \, | \, p_\ba(g.J_m) = 0 \}$ whereas our stratum $\Gamma_{\ba}$ is derived from the {\it locally closed} subscheme
that misses the hypersurface $\{ J \in \hilb^P\bP(V) \, | \, p_\ba(J_m) = 0 \}$ and is contained in the closed subscheme $\cap_{\ba' \succ \ba} \{ p_{\ba'}(J_m) = 0 \}$.  If $g.J$ has initial ideal $I_\ba$, then $J \in H_{I_\ba}^g$, but there are ideals $J \in H_{I_\ba}^g$ whose initial ideal after coordinate change by $g$ differs from $I_\ba$. Hence the Borel open cover does not give information about our stratification in Theorem~\ref{T:gin-decomp}.

\section{Grothendieck-Pl\"ucker embedding is degenerate}\label{S:degenerate}

Retain notations from Section~\ref{S:sch-hilb}. Let $P$ be a non-constant admissible Hilbert polynomial of a graded ideal of $S = k[x_0,\dots,x_n]$, and let $m_o$ denote its Gotzmann number.

\begin{theorem}\label{T:degenerate} The Grothendieck-Pl\"ucker image $\phi_{m} (\hilb^P\bP^n)$ is degenerate for $m > m_0$ unless $P$ is a constant.
\end{theorem}

We first prove the following key lemma.
\begin{lemma}\label{L:revlex} Let $I_m$ be a subspace of $S^mV^*$ generated by an initial reverse lexicographic segment of monomials.
\begin{enumerate}
\item
 For any $l>0$, $S_lI_m$ is also generated by an initial reverse lexicographic segment if and only $\dim_kI_m\ge\binom{n+m-1}{m}$, i.e. if and only if $I_m$ contains $x_{n-1}^m$.

\item Suppose that $I_m$ contains $x_{n-1}^m$. Then
 $\dim_k S^mV^*-\dim_kI_m = \dim_kS^{m+l}V^*- \dim_kS_lI_m$. In particular,  if an ideal $J$ is generated in degree $\le m$ and $J_m = I_m$, then its Hilbert polynomial is a constant. \end{enumerate}

\end{lemma}

\begin{proof} (1) \, One direction is straightforward: Assume that $S_lI_m$ is generated by an initial reverse lexicographic segment. Since $S_lI_m$ contains a monomial divisible by $x_n$, $x_{n-1}^{m+l}$ is also contained in $S_lI_m$. Thus $x_{n-1}^m$ must be contained in $I_m$.

Conversely, assume that $I_m$ contains $x_{n-1}^m$. It is enough to prove that $S_1I_m$ is generated by an initial reverse lexicographic segment. Let $\mu$ be a degree $m$ monomial $M$ not divisible by $x_n$. Then $\mu \succ x_{n-1}^m$ and it follows that $\mu \in I_m$ since $x_{n-1}^m$ is contained in $I_m$ and $I_m$ is generated by a revlex initial segment. In turn, we deduce that every monomial not contained in $S_1I_m$ is divisible by $x_n$.

Consider two monomials $\mu_1, \mu_2 \in S_1I_m$ such that $\mu_1\succ\mu_2$ and $\mu_1\notin S_1I_m$. Then $\mu_1$ is divisible by $x_n$ and since $\mu_2 \prec \mu_1$, $\mu_2$ is also divisible by $x_n$. If $\mu_2\in S_1I_m$, then $\frac{\mu_2}{x_i}\in I_m$ for some $x_i$. The relation $\frac{\mu_2}{x_i}\preceq\frac{\mu_2}{x_n}\preceq\frac{\mu_1}{x_n}$ implies that $\frac{\mu_1}{x_n}\in I_m$ and thus $\mu_1\in S_1I_m$ which is a contradiction. Hence $\mu_2 \not\in S_1I_m$ and this means that $S_1 I_m$ is generated by an initial segment.

\

\noindent (2) \, It suffices to prove that the number of monomials of degree $m$ not contained in $I_m$ is equal to the number of monomials of degree $m+1$ not contained in $S_1I_m$. If $\mu$ is the least monomial in $I_m$ then $x_n\mu$ is the least monomial in $S_1I_m$. So a monomial $\mu'$ of degree $m+1$ is not in $S_1I_m$ iff $x_n\mu\succ\mu'$. The last equation implies that $\mu'$ is divisible by $x_n$, so there is a bijective map  given by $\mu'' \mapsto x_n\mu''$ from the set of monomials of degree $m$ smaller than $\mu$ to the set of monomials of degree $m+1$ smaller than $x_n\mu$.

\end{proof}

An elementary argument shows that if an ideal $I$ is generated by a lex initial segment $I_m$, then $I$ is a lex initial ideal i.e.  $I_{m'}$ is  generated by a lex initial segment for all $m'\ge m$. Moreover, by Macaulay's theorem, given any admissible Hilbert polynomial $P$, one can construct an ideal $I$ whose Hilbert polynomial is $P$ by taking the ideal generated by a suitable initial lexicographic segment. The corollary below states that the opposite holds for revlex: An ideal $\langle W \rangle$ generated by a revlex initial segment $W \subset S_m$ is never an revlex initial ideal, except in the constant Hilbert polynomial case.

\begin{coro}\label{C:revlex} Let $J$ be a graded ideal of $S$. If $J_m$ is generated by a reverse lexicographic initial segment for some $m > 0$, then $J$ has a constant Hilbert polynomial unless $J$ is generated in degrees $\ge m$.
\end{coro}
\begin{proof} If $J$ is not generated in degrees $\ge m$ i.e. $J_{m'} \ne 0$ for some $m' < m$, then $J_m \supset S_{m-m'}J_{m'} \supset x_n^{m-m'} J_{m'}$. Since $J_m$ is generated by a reverse lexicographic initial segment, $J_m \ni x_{n-1}^m$. The assertion now follows due to Lemma~\ref{L:revlex}.(2).
\end{proof}

\begin{proof}[Proof of Theorem~\ref{T:degenerate}] We will prove that for all $m>m_0$, $\phi_{m} (H)$ is degenerate where $H$ is an irreducible component of $\hilb^P\bP^n$. Let $H$ be an irreducible component of $\hilb^P \bP(V)$. Let $\succ$ be the degree reverse lexicographic order and $m> m_o$ be an integer.
Let $N = \dim_k S_m$ and $d  =  N - P(m)$.
We consider $S_m \simeq S^mV^*$ as an $N$-dimensional $k$-vector space with the basis consisting of degree $m$ monomials ordered by $\succ$.  Then $\bigwedge^d S_m$ is the exterior product of $S_m$ with the partially ordered basis consisting of exterior products of degree $m$ monomials.

Let $\ba^{*, m} = \{\a(1), \dots, \a(d)\}$ be the maximal index set so that $x^{\a(1)}\wedge \cdots \wedge x^{\a(d)}$ is the maximal basis element of $\bigwedge^d S_m$, and let $p_{\ba^{*,m}}$ denote the corresponding Pl\"ucker coordinate.
Then $C_{\ba^{*,m}} = \{p_{\ba^{*,m}} \ne 0 \}$ is the big open cell of $\Gr_dS_m$, and its complement $\{p_{\ba^{*,m}} = 0 \}$ defines an ample divisor, namely,  the pull-back $\phi_m^*\cO_{\bP(\bigwedge^d S_m)}(+1)$ of the hyperplane divisor.

We will show that $\phi_m(H)\subset\{p_{\ba^{*,m}} = 0 \}$ using Corollary~\ref{C:revlex}. If $\phi_m(H)\not\subset\{p_{\ba^{*,m}} = 0 \}$, then there exists $I\in\phi^{-1}_m(C_{\ba^{*,m}})\cap H$ such that $[\ii_\prec(I)]_m=[I_{\ba^{*,m}}]_m$. Since $H$ is closed and $\ii_\prec(I)$ is the flat limit of an isotrivial family whose generic object is $I$,  the component $H$ contains  $J:=\ii_\prec(I)$. Since $m > m_0$ and $J$ is generated in degree $\le m_0$,  we have $S_{m-m_0}J_{m_0}=J_m$. Since $[J]_m$ is generated by an initial reverse lexicographic segment,   we may apply Corollary~\ref{C:revlex} to conclude that  the Hilbert Polynomial of $I$ is a constant, contradicting the assumption. 

So for each irreducible component $H$ of $\hilb^P \bP(V)$, $\phi_m(H)$ is contained in $\{p_{\ba^{*,m}} = 0\}$ and thus $\phi_m(\hilb^P \bP(V))$ is contained in $\{p_{\ba^{*,m}} = 0\}$ for $m>m_0$.

\end{proof}

%\begin{remark} Bertone, Lella and Roggero considers covering the Hilbert schemes by
%with $C_{\ba}$ (and their images under linear coordinate changes) such that $I_{\ba}$
%are Borel fixed ideals \cite{BLR}. They give a constructive proof and carry out explicit computations in non-trivial cases.
%\end{remark}

As an example, we consider the Hilbert scheme of degree $d$ hypersurfaces of $\bP^n $. We let $S = k[x_0,\dots, x_n]$ be the homogeneous coordinate ring of $\bP^n$ and let $V^* = S_1$ as before.  Hypersurfaces have Hilbert polynomial $P(m) = \binom{n+m}m - \binom{n+m-d}{m-d}$, and the Hilbert scheme $\hilb^P \bP^n$ is naturally identified with
$
\bP(H^0(\cO_{\bP^n}(d)))
$
and $\phi_d$ is an isomorphism. On the other hand, consider the image under $\phi_{d+1}$. For any $[I] \in \hilb^P\bP^n$,  $\phi_{d+1}([I]) \in \bP(\bigwedge^{n+1} S_{d+1})$ is determined by
\[
\wedge^{d+1} I := x_0f \wedge x_1 f\wedge \cdots \wedge x_n f
\]
where $f$ is a homogeneous degree $d$ polynomial generating $I$. Therefore every monomial appearing in the wedge product $\wedge^{d+1}I$ has each variable $x_0, \dots, x_n$ with a positive exponent. It follows that $\phi_{d+1}(\hilb^P\bP^n)$ is contained in the hyperplane cut out by $p_{\ba} = 0$ where $\ba = (\ba(1), \dots, \ba(n+1))$ is chosen such that  $\{x^{\ba(1)}, \dots, x^{\ba(n+1)}\}$ are monomials of degree $d+1$ in $x_0,\dots, x_{n-1}$: Note that such $\ba$ exists since the number of degree $d+1$ monomials in $x_0,\dots, x_{n-1}$ is $\binom{n+d}{d+1}$ which is larger than $n+1$ for any $n \ge 2$ and $d \ge 1$. Hence we conclude that $\phi_{d+1}(\hilb^P\bP^n)$ is degenerate. Since the Gotzmann number of the Hilbert polynomial $P(m)$ is $d$, this is also checked from the Theorem~\ref{T:degenerate}.

\bibliographystyle{alpha}
\bibliography{sch-decomp-hilb}
\end{document}